\documentclass[twocolumn]{autart}

\usepackage{cite}
\usepackage{amsmath,amssymb,amsfonts}
\usepackage{algorithmic}
\usepackage[hidelinks]{hyperref}
\usepackage{color}
\usepackage{textcomp}
\usepackage{graphicx}
\usepackage{mathrsfs}

\usepackage{natbib}

\newtheorem{theorem}{Theorem}[section]

\newtheorem{definition}[theorem]{Definition}
\newtheorem{corollary}[theorem]{Corollary}
\newtheorem{proposition}[theorem]{Proposition}

\newtheorem{remark}[theorem]{Remark}

\newcommand{\setdef}[2]{\{#1 \, : \, #2\}}

\newcommand{\until}[1]{\{1,\dots,#1\}}

\newcommand{\Ac}{\mathcal{A}}

\newcommand{\Hc}{\mathcal{H}}
\newcommand{\Fc}{\mathcal{F}}

\newcommand{\real}{\mathbb{R}}

\newcommand{\Cc}{\mathcal{C}}
\newcommand{\Dc}{\mathcal{D}}
\newcommand{\Sc}{\mathcal{S}}
\newcommand{\Pc}{\mathcal{P}}
\newcommand{\Bc}{\mathcal{B}}
\newcommand{\Gc}{\mathcal{G}}
\newcommand{\Vc}{\mathcal{V}}
\newcommand{\Wc}{\mathcal{W}}
\newcommand{\Ic}{\mathcal{I}}
\newcommand{\Kc}{\mathcal{K}}

\newcommand{\Lc}{\mathcal{L}}

\newcommand{\argmin}[2] {\mathrm{arg}\min_{#1}#2}

\DeclareSymbolFont{bbold}{U}{bbold}{m}{n}
\DeclareSymbolFontAlphabet{\mathbbold}{bbold}

\newcommand{\norm}[1]{\lVert#1\rVert}

\renewcommand{\theenumi}{(\roman{enumi})}

\newcommand\xqed[1]{%
  \leavevmode\unskip\penalty9999 \hbox{}\nobreak\hfill
  \quad\hbox{#1}}
\newcommand\demo{\xqed{$\bullet$}}

\newcommand{\longthmtitle}[1]{\mbox{}\emph{(#1):}}

 \allowdisplaybreaks

\begin{document}

\begin{frontmatter}
  
  \title{\bf Feasibility and Regularity Analysis of Safe Stabilizing
    Controllers under Uncertainty\thanksref{footnoteinfo}}

  \thanks[footnoteinfo]{This work was supported by
    ARL-W911NF-22-2-0231.}
  
  \author[UCSD]{Pol Mestres}\ead{pomestre@ucsd.edu}%
  \quad \author[UCSD]{Jorge Cortes}\ead{cortes@ucsd.edu}%
  \address[UCSD]{Department of Mechanical and Aerospace Engineering,
    University of California, San Diego.}


\begin{abstract}%
  This paper studies the problem of safe stabilization of
  control-affine systems under uncertainty. Our starting point is the
  availability of worst-case or probabilistic error descriptions for
  the dynamics, a control barrier function (CBF) and a control
  Lyapunov function (CLF). These descriptions give rise to
  second-order cone constraints (SOCCs) whose simultaneous
  satisfaction guarantees safe stabilization. We study the feasibility
  of such SOCCs and the regularity properties of various controllers
  satisfying them.
\end{abstract}

\end{frontmatter}

\section{Introduction}

The last years have seen a dramatic increase in the deployment of
robotic systems in diverse areas like home automation and autonomous
driving. In these applications, it is critical that robots satisfy
simultaneously safety and performance specifications in the presence
of model uncertainty.  Controllers that achieve these goals are
usually defined using tools from stability analysis and Lyapunov
theory. However, this raises several challenges. Among them, we
highlight understanding the level of uncertainty about the model that
can be tolerated while still being able to meet safety and stability
requirements, the characterization of the regularity properties of the
controller, and the identification of suitable conditions to ensure
them in order to be implementable in real-world scenarios.



\emph{Literature Review:} Control Lyapunov functions
(CLFs)~\citep{ZA:83} are a well-established tool for designing
stabilizing controllers for nonlinear systems. More recently, control
barrier functions (CBFs)~\citep{PW-FA:07,ADA-SC-ME-GN-KS-PT:19} have
been introduced as a tool to render a certain predefined set safe. If
the system is control affine, the CLF and CBF conditions can be
incorporated in a quadratic program (QP)~\citep{ADA-XX-JWG-PT:17} that
can be efficiently solved online. Several recent
  works~\citep{WSC-DVD:21,PO-JC:19-cdc,MFR-APA-PT:21,PM-JC:23-csl}
  study the feasibility of such CLF-CBF QP, as well as different
  explicit control designs based on it. However, this design assumes
complete knowledge of the dynamics and safe set.
Several recent papers have proposed alternative formulations of the
CLF-CBF QP for systems with uncertainty or learned
dynamics.
For a particular class of uncertainties,~\cite{MJ:18}
shows that the robust control design problem can still be posed as a QP.
However, imperfect knowledge of the system dynamics or safety
constraints often transforms the affine-in-the-input inequalities
arising from CBFs and CLFs into second-order cone constraints
(SOCCs). The
papers~\citep{FC-JJC-BZ-CJT-KS:21,FC-JJC-BZ-CJT-KS:21-acc} leverage
Gaussian Processes (GPs) to learn the system dynamics from data and
show that the mean and variance of the estimated GP can be used to
formulate two SOCCs whose pointwise satisfaction implies safe
stabilization of the true system with a prescribed
probability. However, during the control design stage, the SOCC
associated to stability is relaxed and hence the resulting controller
does not have stability guarantees. In the case where worst-case error
bounds for the dynamics and the CBF are
known,~\citep{KL-VD-ML-JC-NA:22-ral,KL-CQ-JC-NA:21-ral} show how the
satisfaction of two SOCCs can yield a safe stabilizing controller
valid for all models consistent with these error
bounds.~\citep{KL-YY-JC-NA:23-acc} use the framework of
distributionally robust optimization to formulate a second-order
convex program that achieves safe stabilization for systems with
parametric uncertainty with a finite number of samples. Critically,
as opposed to the uncertainty-free case, where
conditions for the simultaneous satisfaction of the CLF and CBF
conditions are available, cf.~\citep{PM-JC:23-csl}, these works lack
guarantees on the simultaneous feasibility of these SOCCs and the
regularity of controllers satisfying them. As a result, the proposed
controllers might be undefined in practice, resulting in deadlock or
unsafe, unstable, or discontinuous system behaviors. The
identification of conditions under which feasibility guarantees hold
is precisely the main subject of this paper. Finally, the
papers~\citep{VD-MJK-MF-NA:23,FC-JJC-WJ-BZ-CJT-KS:22} utilize online
data to improve the estimates of the dynamics and synthesize (also via
SOCCs) less conservative controllers.

%



\emph{Statement of Contributions:} We study the problem of safe
stabilization of control-affine systems under uncertainty. We consider
two scenarios for the estimates of the dynamics and safe set: either
worst-case error bounds or probabilistic descriptions in the form of
Gaussian Processes (GPs) are available. In both cases, the problem of
designing a safe stabilizing controller can be reduced to satisfying
two SOCCs at every point in the safe set. Our first contribution
consists of giving conditions for the feasibility of each pair of
SOCCs.  The first result is a sufficient condition that requires a
bound on the norm of a safe and stabilizing controller and quantifies
what model errors are tolerable while still being able to find a
controller that guarantees safe stabilization. Our
  second result is a sufficient condition that does not require
  knowledge of such bound and consists of finding a root of a scalar
  nonlinear equation.
Our third contribution consists in giving different regularity
properties for controllers satisfying a set of SOCCs.  First we show
that if each pair of SOCCs is feasible, then there exists a smooth
safe stabilizing controller.  Second, we show that the minimum-norm
controller satisfying each pair of SOCCs is
point-Lipschitz. Third, we provide a universal formula for
satisfying a single SOCC and hence achieving either safety or
stability. We illustrate our results in the safe stabilization of a
planar system.

\section{Preliminaries}

This section presents preliminaries on control Lyapunov and barrier
functions, and safe stabilization using worst-case and probabilistic
estimates of the dynamics.

\subsection{Notation}
We use the following notation. We denote by $\mathbb{Z}_{>0}$,
$\real$, and $\real_{\geq0}$ the set of positive integers, real, and
nonnegative real numbers, resp. We denote by $\mathbf{0}_{n}$ the
$n$-dimensional zero vector, and by $\mathbb{I}_{m}$ the $m\times m$
identity matrix. We write $\text{int}(\Sc)$ and $\partial\Sc$ for the
interior and the boundary of the set $\Sc$, resp. Given
$x\in\real^{n}$, $\norm{x}$ denotes the Euclidean norm of $x$. Given
$f:\real^{n}\to\real^{n}$, $g:\real^{n}\to\real^{n\times m}$ and a
smooth function $W:\real^{n}\to\real$, $L_{f}W:\real^{n}\to\real$
(resp. $L_{g}W:\real^{n}\to\real^{m}$) denotes the Lie derivative of
$W$ with respect to $f$ (resp. $g$), that is $L_{f}W=\nabla W^{T} f$
(resp. $\nabla W^{T} g$). We use $\Gc\Pc(\mu(x),K(x,x'))$ to denote a
Gaussian Process distribution with mean function $\mu(x)$ and
covariance function $K(x,x')$.  We denote by $\Cc^{l}(A)$ the set of
$l$-times continuously differentiable functions on an open set
$A\subseteq\real^{n}$. A function $\beta:\real_{\geq0}\to\real$ is of
class $\Kc$ if $\beta(0)=0$ and $\beta$ is strictly increasing. If
moreover $\lim_{t\to\infty}\beta(t)=\infty$, then $\beta$ is of class
$\Kc_{\infty}$. A function $V:\real^{n}\to\real$ is positive definite
if $V(0)=0$ and $V(x)>0$ for all $x\neq0$. $V$ is proper in a set
$\Gamma$ if the set $\setdef{x\in\Gamma}{V(x)\leq c}$ is compact for
any $c\geq0$. A set $\Cc\subseteq\real^{n}$ is forward invariant under
the dynamical system $\dot{x}=f(x)$ if any trajectory with initial
condition in $\Cc$ at time $t=0$ remains in $\Cc$ for all positive
times. A set $\Cc$ is safe for $\dot{x}=f(x,u)$ if there exists a
locally Lipschitz control $k:\real^{n}\to\real^{m}$ such that $\Cc$ is
forward invariant for $\dot{x}=f(x,k(x))$. Given $m\times n$ matrix
$A$ and two integers $i, j$ such that $1\leq i < j\leq n$, $A_{i:j}$
denotes the $m\times(j-i+1)$ matrix obtained by selecting the columns
from $i$ to $j$ of $A$, and $\sigma_{\max}(A)$ denotes the maximum
singular value of $A$. The image of $A$ is defined as
$\text{Im}(A)=\setdef{y\in\real^m}{\exists \ x\in\real^n \ \text{s.t.}
  \ y=Ax}$.  We denote by
$\Bc_{r}(p)=\setdef{y\in\real^n}{\norm{y-p}<r}$. Given
$A\in\real^{q\times n}$, $b\in\real^{q}$, $c\in\real^n$, $d\in\real$,
the inequality $\norm{Ax+b}\leq c^{T}x+d$ is a second-order cone
constraint (SOCC) in the variable $x\in\real^n$. A
function $f:\real^n\to\real^q$ is point-Lipschitz at a point
$x_0\in\real^n$ if there exists a neighborhood $\Vc$ of $x_0$ and a
constant $L>0$ such that $\norm{f(x)-f(x_0)}\leq L\norm{x-x_0}$ for
all $x\in\Vc$.

\subsection{Control Lyapunov and Barrier Functions}

Consider the control-affine system
\begin{align}\label{eq:control-affine-sys}
  \dot{x}=f(x)+g(x)u ,
\end{align}
where $f:\real^{n}\to\real^{n}$ and $g:\real^{n}\to\real^{n\times m}$
are locally Lipschitz functions, with $x\in\real^{n}$ the state and
$u\in\real^{m}$ the input.  We assume without loss of generality that
$f(\mathbf{0}_n)=\mathbf{0}_{n}$. 

\begin{definition}\longthmtitle{Control Lyapunov
    Function~\citep{EDS:98}}\label{def:clf}
  Given a set $\Gamma\subseteq\real^{n}$, with
  $\mathbf{0}_n\in\Gamma$, a continuously differentiable function
  $V:\real^{n}\to\real$ is a CLF on $\Gamma$ for the
  system~\eqref{eq:control-affine-sys} if it is proper in $\Gamma$,
  positive definite, and there exists a continuous positive definite
  function $W:\real^{n}\to\real$ such that, for each
  $x\in\Gamma\backslash \{ \mathbf{0}_n \}$, there exists a control $u\in\real^{m}$
  satisfying
  \begin{align}\label{eq:clf-ineq}
    L_fV(x)+L_gV(x)u \leq -W(x).
  \end{align}
\end{definition}
\smallskip

A Lipschitz controller $k:\real^{n}\to\real^{m}$ such that $u=k(x)$
satisfies~\eqref{eq:clf-ineq} for all
$x\in\Gamma\backslash\{ \mathbf{0}_n \}$ makes the origin of the
closed-loop system asymptotically stable~\citep{EDS:98}. Hence, CLFs
enable to guarantee asymptotic stability.

\begin{definition}\longthmtitle{Robust Control Barrier Function~\citep[Definition
    6]{MJ:18}}\label{def:robust-safety}
  Let $\Cc \subset \real^n$ and $h:\real^{n}\to\real$ be a continuously
  differentiable function such that
  \begin{subequations}\label{eq:safe-set}
    \begin{align}
      \Cc &=\setdef{x\in\real^n}{h(x)\geq0},
      \\
      \partial\Cc &=\setdef{x\in\real^n}{h(x)=0}.
    \end{align}
  \end{subequations}
  Given $\eta>0$, $h$ is an $\eta$-robust CBF if
  there exists a class $\Kc_{\infty}$ function $\alpha$ such that for
  all $x\in\Cc$, there exists $u\in\real^{m}$ with
  \begin{align}\label{eq:robust-safety}
    L_fh(x) + L_gh(x)u +\alpha(h(x)) \geq \eta.
  \end{align}
\end{definition}
\smallskip

When $\eta=0$, this definition reduces to the notion of
CBF~\citep[Definition 2]{ADA-SC-ME-GN-KS-PT:19}, and the inequality
reduces~to
\begin{align}\label{eq:cbf-ineq}
  L_fh(x) + L_gh(x)u +\alpha(h(x)) \geq 0 .
\end{align}
Note that all robust CBFs are CBFs.  A Lipschitz controller
$k:\real^{n}\to\real^{m}$ such that $u=k(x)$
satisfies~\eqref{eq:cbf-ineq} for all $x\in\Cc$ makes $\Cc$ forward
invariant~\citep[Theorem 2]{ADA-SC-ME-GN-KS-PT:19}. Hence, CBFs enable
to guarantee safety.

\begin{remark}\longthmtitle{Alternative CLF and CBF
    conditions}\label{rem:slack-clf}
  {\rm 
  Without loss of generality, if $V$ is a CLF on an open set $\Gamma$,
  we can assume that there exists a positive definite function $S$ such
  that, for all $x\in\Gamma$, there is $u\in\real^{m}$ with
  \begin{align}\label{eq:clf-slack}
    L_{f}V(x)+L_{g}V(x)u+W(x)\leq-S(x).
  \end{align}
  This is because if~\eqref{eq:clf-ineq} holds, we can always define
  $\tilde{W}(x):=\frac{1}{2}W(x)$ and let $\tilde{W}$ play the role of $W$ in~\eqref{eq:clf-slack} and take $S(x):=\frac{1}{2}W(x)$. Similarly, if $h$ is an $\eta$-robust
  CBF, then there exists a class $\Kc_{\infty}$ function $\zeta$ such
  that for all $x\in\Cc$, there is $u\in\real^{m}$ with
  \begin{align}\label{eq:cbf-slack}
    L_fh(x)+L_gh(x)u+\alpha(h(x))\geq \eta+\zeta(h(x)).   \demo
  \end{align}}
\end{remark}


\subsection{Robust and Probabilistic Safe
  Stabilization}\label{sec:rob-prob-safe-stab}

We are interested in the design of controllers that ensure stability
and safety in the presence of uncertainty.  We assume that the maps
$f$, $g$ in~\eqref{eq:control-affine-sys} and the CBF $h$ and its
gradient $\nabla h$ are unknown. We also assume that a CLF~$V$ for
the true system is unknown.  Instead, estimates of $f$, $g$, $h$,
$\nabla h$, $V$, and $\nabla V$ (denoted $\hat{f}$, $\hat{g}$,
$\hat{h}$, $\widehat{\nabla h}$, $\hat{V}$, and $\widehat{\nabla V}$
resp.) are available.

\begin{remark}\longthmtitle{Lyapunov function search under
    uncertainty}\label{re:exact-Lyapunov-uncertainty}
  {\rm We assume that $f$, $g$ and $h$ are only approximately known
    because, in practice, the dynamic model and safety constraints are
    often obtained using noisy sensor data and simplified models,
    which leads to estimation errors. The construction of CLFs for
    these approximations in turn leads to approximations of the CLF
    for the true system.  However, there are techniques to find CLFs
    for uncertain systems including sum-of-squares~\citep{AAA-AM:16},
    which is limited to polynomial systems but provides known error
    bounds,~\citep{AJT-VDD-HML-YY-ADA:19}, which describes a method
    that only requires knowledge of the degree of actuation,
    and~\citep{KL-YY-JC-NA:23-l4dc}, which uses ideas from
    distributionally robust optimization. All these works seek to find
    a CLF that is valid for all systems compatible with the given
    uncertainty. In our treatment, we only require $\hat{V}$ and
    $\widehat{\nabla V}$ to be within some error bounds of a true CLF
    and its gradient, respectively, but if a true CLF is known (by
    using for instance the techniques in the given references), these
    error bounds can be taken as identically zero. \demo}
\end{remark}

We consider two types of models for the errors between the estimates
and the true quantities. First, for $x\in\real^{n}$, consider
worst-case error bounds as follows:
\begin{align*}
  \norm{f(x)-\hat{f}(x)}\leq e_f(x), \ \norm{g(x)-\hat{g}(x)}\leq
  e_g(x),
  \\ 
  |h(x)-\hat{h}(x)| \leq e_h(x), \ \norm{\nabla
  h(x)-\widehat{\nabla h}(x)}\leq e_{\nabla h}(x),
  \\
  |V(x)-\hat{V}(x)| \leq e_V(x), \ \norm{\nabla V(x)-\widehat{\nabla
  V}(x)} \leq e_{\nabla V}(x). 
\end{align*}
Since the exact dynamics, the CBF and CLF are unknown, one can not certify the
inequalities~\eqref{eq:clf-ineq} and~\eqref{eq:cbf-ineq}
directly. Instead, using the error bounds above, define
\begin{align*}
  a_V(x) &=e_{\nabla V}(x)e_g(x) + e_{\nabla V}(x)\norm{\hat{g}(x)} +
           \norm{\widehat{\nabla V}(x)}e_g(x) ,
  \\ 
  b_V(x) &=-\widehat{\nabla V}(x)^T \hat{g}(x),
  \\
  c_V(x) &=-e_{\nabla V}(x)e_f(x)-e_{\nabla
           V}(x)\norm{\hat{f}(x)}-\norm{\widehat{\nabla V}(x)}e_f(x)
  \\
         &\quad  -\widehat{\nabla V}(x)^T \hat{f}(x) - W(x),
  \\
  a_h(x) &=e_{\nabla h}(x)e_g(x)+e_{\nabla
           h}(x)\norm{\hat{g}(x)}+\norm{\widehat{\nabla h}(x)}e_g(x),
  \\
  b_h(x) &=\widehat{\nabla h}(x)^T \hat{g}(x),
  \\
  c_h(x) &=-e_{\nabla h}(x) e_f(x)-e_{\nabla
           h}(x)\norm{\hat{f}(x)}-\norm{\widehat{\nabla h}(x)}e_f(x)
  \\
         & \quad \quad + \widehat{\nabla h}(x)^T
           \hat{f}(x)+\alpha(\hat{h}(x)-e_h(x)).
\end{align*}
According to~\cite[Proposition V.I]{KL-VD-ML-JC-NA:22-ral}, if the two
(state-dependent) SOCCs (in $u$):
\begin{subequations}\label{eq:socc-worst-case}
  \begin{align}
    a_V(x)\norm{u}
    & \leq
      b_V(x)u+c_V(x) , \label{eq:socc-clf-worst-case}
    \\
    a_h(x)\norm{u} & \leq b_h(x)u+c_h(x), \label{eq:socc-cbf-worst-case}
  \end{align}
\end{subequations}
are satisfied for all $x\in\Cc$, then~\eqref{eq:clf-ineq}
and~\eqref{eq:cbf-ineq} hold for all $x\in\Cc$. This result provides a
way of designing controllers that simultaneously
satisfy~\eqref{eq:clf-ineq} and~\eqref{eq:cbf-ineq}.

Second, suppose that GP estimates are available for the following
quantities~\citep{FC-JJC-BZ-CJT-KS:21}:
\begin{align*}
  \Delta_V(x,u)
  &=L_f V(x)+L_g V(x)u - \widehat{\nabla V}(x)^T (\hat{f}(x)+\hat{g}(x)u),
  \\
  \Delta_h(x,u)
  &= L_fh(x)+L_gh(x)u+\alpha(h(x)) -\widehat{\nabla
    h}(x)^T \hat{f}(x)
  \\
  &
    \quad -\widehat{\nabla h}(x)^T \hat{g}(x)u -\alpha(\hat{h}(x)).
\end{align*}
We further assume that if $\Hc$ is the Reproducing Kernel Hilbert
Space (RKHS,~\citep[Section 2.1]{NS-AK-SMK-MS:10-arxiv}) with respect
to which the GP estimates of $\Delta_{V}$ and $\Delta_{h}$ have been
derived, then $\Delta_{V}$ and $\Delta_{h}$ have bounded RKHS norm
with respect to $\Hc$.  Let $\mu_{V}(x,u)$ and $s^{2}_{V}(x,u)$ denote
the mean and variance, resp., of the GP prediction of $\Delta_{V}$,
which we assume affine and quadratic in $u$, resp. Therefore, there
exist $\gamma_{V}(x):\real^n\to\real^{m+1}$ and
$G_{V}(x) \in \real^{(m+1)\times(m+1)}$ such that
\begin{gather*}
  \mu_V(x,u)=\gamma_V(x)^T \begin{bmatrix} 1 \\ u \end{bmatrix}, \quad
  s_V(x,u)= \norm{G_V(x) \begin{bmatrix} 1 \\ u \end{bmatrix}}_2.
\end{gather*}
For the GP prediction of $\Delta_{h}$, let $\gamma_{h}(x)$, and
$G_{h}(x)$ be defined analogously.  Since the exact dynamics, the CBF and CLF
are unknown, one cannot certify the inequalities~\eqref{eq:clf-ineq}
and~\eqref{eq:cbf-ineq}. However, for $\delta\in(0,1)$, and using the
GP predictions, define
\begin{align*}
  Q_V(x) &=\beta(\delta) G_{V,2:(m+1)}(x)\in\real^{(m+1)\times m},
  \\
  r_V(x) &=\beta(\delta) G_{V,1}(x)\in\real^{(m+1)\times 1},
  \\
  b_V(x) &=-\widehat{\nabla
           V}(x)^T\hat{g}(x)-\gamma_{V,2:(m+1)}^T(x)\in\real^{1\times 
           m},
  \\
  c_V(x) &=-\widehat{\nabla V}(x)^T\hat{f}(x)-W(x)-\gamma_{V,1}(x)\in\real,
\end{align*}
and similarly $Q_{h}$, $r_{h}, b_{h}$ and $c_{h}$ (the exact form of
$\beta(\delta)$ is given
in~\citep[Theorem~2]{FC-JJC-BZ-CJT-KS:21-acc}). Then, according
to~\citep[Section~IV]{FC-JJC-BZ-CJT-KS:21}, if the two SOCCs
\begin{subequations}\label{eq:socc-gp}
  \begin{align}
    \norm{Q_V(x)u+r_V(x)} & \leq b_V(x)u+c_V(x),~\label{eq:socc-clf-gp}
    \\
    \norm{Q_h(x)u+r_h(x)} & \leq b_h(x)u+c_h(x),~\label{eq:socc-cbf-gp}
  \end{align}
\end{subequations}
are satisfied for all $x\in\Cc$, then~\eqref{eq:clf-ineq}
and~\eqref{eq:cbf-ineq} each hold for all $x\in\Cc$ with probability at
least $1-\delta$.

\begin{remark}\longthmtitle{General form of
    SOCCs}\label{rem:general-form-soccs}
  {\rm 
  By taking
  \begin{align*}
    Q_V(x)=a_V(x)\begin{pmatrix} \mathbb{I}_m
      \\
      \mathbf{0}_{m}^T \end{pmatrix}, \ r_V(x)=\textbf{0}_{m+1} 
  \end{align*}
  in~\eqref{eq:socc-clf-gp}, we
  obtain~\eqref{eq:socc-clf-worst-case}. Hence, in the following, we
  derive the results for SOCCs of the most general
  form~\eqref{eq:socc-gp}.  \demo }
\end{remark}

\section{Problem Statement}

We consider a control-affine system of the
form~\eqref{eq:control-affine-sys} and a safe set $\Cc$ of the
form~\eqref{eq:safe-set} with $f$, $g$, $h$, and $\nabla h$
unknown. We assume that $h$ is a CBF of $\Cc$ and
$\frac{\partial h}{\partial x}(x)\neq0$ for all $x\in\partial\Cc$.
By~\citep[Theorem 2]{ADA-SC-ME-GN-KS-PT:19}, this implies that $\Cc$
is safe. However, since the true $h$ is unknown, a safe controller is
not readily computable.  We further assume that $V$ is an unknown CLF
on an open set containing the origin.  We suppose that either
worst-case or probabilistic descriptions of the dynamics, the CBF and
the CLF are available, as described in
Section~\ref{sec:rob-prob-safe-stab}.

Given this setup, our goals are to (i) derive conditions that ensure
the feasibility of the pair of robust
stability~\eqref{eq:socc-clf-worst-case} and
safety-\eqref{eq:socc-cbf-worst-case} (resp., probabilistic
stability~\eqref{eq:socc-clf-gp} and safety~\eqref{eq:socc-cbf-gp})
inequalities and, building on this, (ii) design controllers that
jointly satisfy the inequalities pointwise in $\Cc$ and characterize
their regularity properties. The latter is motivated by both
theoretical (guarantee the existence and uniqueness of solutions to
the closed-loop system) and practical (ease of implementation of
feedback control on digital platforms and avoidance of chattering
behavior) considerations.

\section{Compatibility of Pairs of Second-Order Cone
  Constraints}\label{sec:compat}

In this section, we derive sufficient conditions that guarantee the
feasibility of the pairs of inequalities in~\eqref{eq:socc-worst-case}
and in~\eqref{eq:socc-gp}, resp. The next definition
extends the notion of compatibility given in~\citep[Definition 3]{PM-JC:23-csl}
to any set of inequalities.

\begin{definition}\longthmtitle{Compatibility of a set of
    inequalities}\label{def:compat-ineqs}
  Given functions $q_{i}:\real^{n}\times\real^{m}\to\real$ for
  $i\in\{1,\hdots,p\}$, the inequalities
  $q_{i}(x,u)\leq0$, $i\in\{1,\hdots,p\}$ are (strictly) compatible at
  a point $x\in\real^{n}$ if there exists a corresponding
  $u\in\real^{m}$ satisfying all inequalities (strictly). The same
  inequalities are (strictly) compatible on a set $\Gc$ if they are
  (strictly) compatible at every $x\in\Gc$.
\end{definition}

As the estimation errors (resp. the variances) approach zero, the inequalities in~\eqref{eq:socc-worst-case} (resp.~\eqref{eq:socc-gp}) approach~\eqref{eq:clf-ineq}
and~\eqref{eq:cbf-ineq}. If~\eqref{eq:clf-slack} and~\eqref{eq:cbf-slack} are
compatible, the next result provides explicit bounds for the
estimation errors such that~\eqref{eq:socc-clf-worst-case}-\eqref{eq:socc-cbf-worst-case}
and~\eqref{eq:socc-clf-gp}-\eqref{eq:socc-cbf-gp} are strictly
compatible.

\begin{proposition}\longthmtitle{Sufficient condition
      for compatibility  
  given upper bound on the norm of a safe stabilizing controller}\label{prop:small-errors}
Let $h$ be an $\eta$-robust CBF. Let $\tilde{\Cc}$ be a set containing
$\Cc$ such that~\eqref{eq:clf-slack} and~\eqref{eq:cbf-slack} are
compatible on $\tilde{\Cc}$. Let $B:\real^{n}\to\real$ be an upper
bound on the norm of a control satisfying both inequalities.  Suppose
$\alpha$ in~\eqref{eq:cbf-slack} is Lipschitz with constant
$K_{\alpha}$. Let $x\in\Cc$.
  \begin{enumerate}
  \item \label{it:compat-worst-case} If
    \begin{subequations}
    \begin{align}
      \notag
      & \norm{\widehat{\nabla V}(x)} (e_f(x)+e_g(x)B(x)) + e_{\nabla
        V}(x) \Big( \norm{\hat{f}(x)}+
      \\ 
      & e_f(x)+(\norm{\hat{g}(x)}+e_g(x))B(x) \Big)
        <\frac{1}{2}S(x),~\label{eq:worst-case-v} 
      \\ 
      \notag
      & (e_{\nabla h}(x)+\norm{\widehat{\nabla
        h}(x)})(e_f(x)+e_g(x)B(x))+K_{\alpha} e_h(x)
      \\
      & +e_{\nabla h}(x)(\norm{\hat{f}(x)}+\norm{\hat{g}(x)}B(x)) <
        \frac{1}{2}(\eta+\zeta(h(x))) ,~\label{eq:worst-case-h} 
    \end{align}
    \label{eq:small-errors-compat}
  \end{subequations}
  then~\eqref{eq:socc-clf-worst-case}
  and~\eqref{eq:socc-cbf-worst-case} are strictly compatible in a neighborhood of $x$;
\item\label{it:compat-prob} If
  \begin{subequations}\label{eq:small-vars}
    \begin{align}
      \sigma_{\max}(G_V(x))<\frac{S(x)}{2\beta(\delta)\sqrt{1+B^2(x)}},
      \label{eq:gp-v}
      \\ 
      \sigma_{\max}(G_h(x))
      <\frac{\eta+\zeta(h(x))}{2\beta(\delta)\sqrt{1+B^2(x)}},
      \label{eq:gp-h} 
    \end{align}
  \end{subequations}
  then~\eqref{eq:socc-clf-gp} and~\eqref{eq:socc-cbf-gp} are strictly
  compatible in a neighborhood of $x$ with probability at least $1-2\delta$.
  \end{enumerate}
\end{proposition}
\begin{pf}
  \ref{it:compat-worst-case}) Since the
  inequalities~\eqref{eq:small-errors-compat} are strict, there exists
  a neighborhood $\Wc_{x}$ of $x$ such
  that~\eqref{eq:small-errors-compat} holds for all points in
  $\Wc_{x}$. The proof follows by applying the definition of
  $e_{f}, e_{g}, e_{h}, e_{\nabla h}$, $e_{\nabla V}$ given in
  Section~\ref{sec:rob-prob-safe-stab}.

  \ref{it:compat-prob}) First note that since the inequalities
  in~\eqref{eq:small-vars} are satisfied at $x$, there exists a
  neighborhood $\Wc_{x}$ of $x$ such that~\eqref{eq:small-vars} hold
  for all points in $\Wc_{x}$.  Note that~\eqref{eq:socc-clf-gp} can
  be equivalently written as
  \begin{align*}
    \beta(\delta)\norm{G_V(x)
    \begin{bmatrix} 1
      \\
      u\end{bmatrix}}_2
    &\leq
      -\widehat{\nabla V}(x)^T \hat{f}(x) - \gamma_{V,1}(x)-W(x)
    \\
    &\quad -(\widehat{\nabla V}(x)^T \hat{g}(x) + \gamma_{V,2:(m+1)}^T(x))u,
  \end{align*}
  and similarly for~\eqref{eq:socc-cbf-gp}.  Now, note that
  $-\widehat{\nabla V}(x)^T(\hat{f}(x)+\hat{g}(x)u) - \gamma_{V,1}(x)
  - \gamma_{V,2:(m+1)}^T(x) u = -L_fV(x)-L_gV(x)u+\Delta_V(x,u) -
  \gamma_{V,1}(x) - \gamma_{V,2:(m+1)}^T(x) u$, and similarly for the
  safety constraint.  Define then the events
  $\mathcal{E}_{V} = \{|\gamma_V(y)^T\begin{bmatrix}
    1\\u\end{bmatrix}-\Delta_V(y,u)|\leq\beta s_V(y,u), \newline
  \forall y\in\Wc_x, u\in\real^{m}\}$ and
  $\mathcal{E}_{h} = \{|\gamma_h(y)^T\begin{bmatrix}
    1\\u\end{bmatrix}-\Delta_h(y,u)|\leq\beta(\delta) s_h(y,u) , \,
  \forall \ y\in\Wc_x, u\in\real^{m}\}$. By~\citep[Theorem
  6]{NS-AK-SMK-MS:10-arxiv}, $\mathbb{P}(\mathcal{E}_V)\geq 1-\delta$
  and $\mathbb{P}(\mathcal{E}_h)\geq1-\delta$. Therefore,
  $\mathbb{P}(\mathcal{E}_V \cap \mathcal{E}_h) =
  \mathbb{P}(\mathcal{E}_V) +
  \mathbb{P}(\mathcal{E}_h)-\mathbb{P}(\mathcal{E}_V\cup\mathcal{E}_h)\geq
  1-2\delta$.
  Hence, if for all $y\in\Wc_{x}$ we can find $u\in\real^{m}$ satisfying
  \begin{subequations}
    \begin{align}
      L_fh(y) + L_gh(y)u + \alpha(h(y))
      & \geq
        2 \beta(\delta)
        s_h(y,u)
      ,
      \\
      -L_fV(y) - L_gV(y)u-W(y)
      & \geq
        2\beta(\delta)
        s_V(y,u)
      ,
    \end{align}
    \label{eq:compat-beta}
  \end{subequations}
  then~\eqref{eq:socc-clf-gp},~\eqref{eq:socc-cbf-gp} are compatible
  at $\Wc_{x}$ with probability at least $1-2\delta$. Let $u^{*}(x)$
  be a control satisfying~\eqref{eq:clf-slack}-\eqref{eq:cbf-slack}
  with $\norm{u^*(x)}\leq B(x)$. Let us show that $u^{*}(y)$
  satisfies~\eqref{eq:compat-beta} for all $y\in\Wc_{x}$. By using the
  characterization of the matrix norm induced by the Euclidean norm
  in~\citep[Example 5.6.6]{RAH-CRJ:12}, we get
  $\norm{G_V(y) \begin{bmatrix} 1\\u^*(y) \end{bmatrix}}_2\leq
  \sigma_{\max}(G_V(y)) \sqrt{1+B^2(y)}$ and similarly for the safety
  constraint. Using now~\eqref{eq:small-vars}, we deduce that
  $u^{*}(y)$ satisfies~\eqref{eq:compat-beta} for all $y\in\Wc_{x}$.
  $\blacksquare$
\end{pf}

\begin{remark}\longthmtitle{Tightness of conditions for
    SOCC compatibility}\label{rem:small-errors-vars}
  {\rm 
  The assumption that $h$ is an $\eta$-robust CBF makes it
  possible for~\eqref{eq:worst-case-h}
  and~\eqref{eq:gp-h} to be satisfied at $\partial\Cc$.
  If the estimation errors (resp. the variances $s_{V}^2$,
  $s_{h}^{2}$) are zero, then~\eqref{eq:small-errors-compat}
  (resp.~\eqref{eq:small-vars}) is trivially satisfied.  Larger values
  of $S(x)$ and $\zeta(h(x))$, and smaller values of $B(x)$, lead to
  conditions that are easier to satisfy. Closer to the origin, $S(x)$
  becomes smaller, thus making~\eqref{eq:worst-case-v}
  and~\eqref{eq:gp-v} harder to satisfy. In
  fact,~\eqref{eq:worst-case-v} and~\eqref{eq:gp-v} can only be
  satisfied near the origin if knowledge of $\nabla V$ is exact near
  it, cf. Remark~\ref{re:exact-Lyapunov-uncertainty}.  If
  $\zeta(h(x))$ is unknown, a known lower bound for it (e.g., $0$) can
  be used at the expense of more conservativeness.  \demo }
\end{remark}

\begin{remark}\longthmtitle{Computation of upper bound of safe
    stabilizing controller}\label{rem:upper-bound-form-csl}
  {\rm One can obtain~$B$ in Proposition~\ref{prop:small-errors} by
    relying on the expression for a safe stabilizing controller
    provided in~\citep{PM-JC:23-csl}, together with upper and lower
    bounds on the norms of $f$, $g$, $h$, $\nabla h$, $V$, and
    $\nabla V$.  \demo}
\end{remark}

We next provide a sufficient condition for the
  compatibility of~\eqref{eq:socc-gp} which does not require knowledge
  of an upper bound on the norm of a safe stabilizing controller. To
  do so, we first introduce some useful notation.
  Given~\eqref{eq:socc-gp}, define
  $\hat{A}:\real\to\real^{m\times m}$, $\hat{B}:\real\to\real^m$ by
  (note we have dropped the state-dependency in $x$ for brevity)
  \begin{align*}
    \hat{A}(\lambda)
    &:= 2(Q_V^T Q_V-b_V^T b_V) + 2\lambda (Q_h^T
      Q_h - b_h^T b_h),
    \\ 
    \hat{B}(\lambda)
    &:= 2(Q_V^T r_V-b_V^T c_V) + 2\lambda (Q_h^T
      r_h - b_h^T c_h), 
  \end{align*}
  and the set
  $\Fc_0 :=\setdef{\lambda
    \in\real}{\operatorname{det}(\hat{A}(\lambda))\neq0}$.  Let
  $A:\Fc_0 \to\real^{m\times m}$ and $d,\alpha_h,\alpha_V:\Fc_0\to\real$ be
  given by
  \begin{align*}
    A(\lambda)
    & := \hat{A}(\lambda)^{-1},
    \\
    d(\lambda)
    & :=b_h A(\lambda)b_h^T b_V A(\lambda)b_V^T-(b_h
      A(\lambda)b_V^T)^2,
    \\
    \alpha_h(\lambda)
    & := b_h A(\lambda)\hat{B}(\lambda)-c_h,
    \\
    \alpha_V(\lambda)
    &
      := b_V A(\lambda)\hat{B}(\lambda)-c_V.
  \end{align*} 
  Further let
  \begin{align*}
    \Fc_1
    &:=\setdef{\lambda\in\real}{\operatorname{det}(\hat{A}(\lambda))\neq0, \
      d(\lambda)\neq0},
  \\
    \Fc_2
    &:=\setdef{\lambda\in\real}{\operatorname{det}(\hat{A}(\lambda))\neq0, \
      b_V A(\lambda) b_V^T \neq0},
    \\
    \Fc_3
    &:=\setdef{\lambda\in\real}{\operatorname{det}(\hat{A}(\lambda))\neq0,
      \ b_h A(\lambda) b_h^T \neq0}, 
  \end{align*}
  and define $\lambda_{2,0}:\real\to\real$,
  $\{\lambda_{2,i}:\Fc_i\to\real\}_{i=1}^3$,
  $\lambda_{3,0}:\real\to\real$,
  $\{\lambda_{3,i}:\Fc_i\to\real\}_{i=1}^3$,
  and $u_i^*:\Fc_i\to\real^m$ for $i\in\{0,1,2,3\}$ as follows:
  \begin{align*}
    \lambda_{2,i}(\lambda)\!
    &:=\!\!
      \begin{cases}
        0 &\!\!\! \text{if } i=0,
        \\
        \frac{1}{d(\lambda)} b_V A(\lambda) (b_V^T
        \alpha_h(\lambda)-b_h^T \alpha_V(\lambda))
          &\!\!\! \text{if }
          i=1,
        \\
        0 &\!\!\! \text{if } i=2,
        \\
        \frac{\alpha_h(\lambda)}{b_h A(\lambda) b_h^T}
          &\!\!\! \text{if
            } i=3,
      \end{cases}
    \\
    \lambda_{3,i}(\lambda)\!
    &:=\!\!
      \begin{cases}
        0 &\!\!\! \text{if } i=0,
        \\
        \frac{1}{d(\lambda)}(-b_V \alpha_h(\lambda)+b_h
        \alpha_V(\lambda))A(\lambda) b_h^T &\!\!\! \text{if } i=1,
        \\
        \frac{\alpha_V(\lambda)}{b_V A(\lambda) b_V^T} &\!\!\! \text{if } i=2,
        \\
        0 &\!\!\! \text{if } i=3,
      \end{cases}
    \\
    u_i^*(\lambda)
    &
      :=A(\lambda)(\lambda_{2,i}(\lambda)b_h^T +
      \lambda_{3,i}(\lambda)b_V^T-\hat{B}(\lambda)).
  \end{align*}
  We are now ready to state the result.

\begin{proposition}\longthmtitle{Sufficient condition for
    compatibility without knowledge of upper bound on the norm of a
    safe stabilizing controller}\label{prop:nec-suff-cond-compat}
  Let the functions $g_h,g_V:\real^m\to\real$ be given by
  {\small \vspace*{-5ex}
  \begin{align*}
    g_h(u)
    \! & = \! u^T(Q_h^TQ_h \! - \! b_h^T b_h)u \! + \! 2(r_h^T
          Q_h \! - \! b_h c_h)u \!
          + \! \norm{r_h}^2 \! - \! c_h^2,
    \\
    g_V(u) \!
        &= \! u^T(Q_V^T Q_V \! - \! b_V^T b_V)u \! + \! 2(r_V^T
          Q_V \! - \! b_V c_V)u \!
          + \! \norm{r_V}^2 \! - \! c_V^2,
  \end{align*}
} and define the functions $\{\eta_i:\Fc_i\to\real\}_{i=0}^3$ by
$\eta_i(\lambda)=\lambda g_h(u_i^*(\lambda))$.  Further consider the
constraints
\begin{align}\label{eq:opt-pb-licq}
  g_h(u)\leq0, \quad -b_h u -c_h \leq 0, \quad -b_V u -c_V \leq 0.
\end{align}
Then,~\eqref{eq:socc-gp} are compatible if there is
$i\in\{0,1,2,3\}$ such that there exists a non-negative root
$\lambda_i^*\in\Fc_i$ of $\eta_i$ such that
$\lambda_{2,i}(\lambda_i^*)\geq0$,
$\lambda_{3,i}(\lambda_i^*)\geq0$,
$g_V(u_i^*(\lambda_i^*))\leq0$, $g_h(u_i^*(\lambda_i^*))$,
the constraints in~\eqref{eq:opt-pb-licq} at $u_i^*(\lambda_i^*)$
are satisfied, and the gradients of the active constraints
in~\eqref{eq:opt-pb-licq} are linearly independent.
\end{proposition}
\begin{pf}
  Let
  \begin{align}\label{eq:proof-nec-suf-cond-pb}
    \sigma
    &:=\min_{u\in\real^m} g_V(u)
    \\
    \notag
    &\quad \text{s.t.} \ g_h(u)\leq0, \quad -b_h u -c_h \leq0, \quad
      -b_V u -c_V \leq 0. 
  \end{align}
  By~\citep{FC-JJC-BZ-CJT-KS:21},~\eqref{eq:socc-gp} are compatible if
  and only if $\sigma\leq0$.  The result now follows by applying the
  KKT conditions to Problem~\eqref{eq:proof-nec-suf-cond-pb}.  The
  condition that the gradients of the active constraints
  in~\eqref{eq:opt-pb-licq} are linearly independent
  guarantees that Linear Independence Constraint Qualification (cf.~\cite[Definition
  2.4]{GS:18}) holds at the optimizer
  of~\eqref{eq:proof-nec-suf-cond-pb}. Hence, the optimizer
  of~\eqref{eq:proof-nec-suf-cond-pb} satisfies the KKT conditions
  of~\eqref{eq:proof-nec-suf-cond-pb}, cf.~\cite[Theorem
  5.33]{NA-AE-MP:20}.  Let then
  $\Lc(u,\lambda_1,\lambda_2,\lambda_3)=g_V(u)+\lambda_1 g_h(u) +
  \lambda_2(-b_h u-c_h)+\lambda_3(-b_Vu-c_V)$ be the Lagrangian
  of~\eqref{eq:proof-nec-suf-cond-pb}.  The stationarity condition
  $\nabla_u \Lc(u,\lambda_1,\lambda_2,\lambda_3)=0$ implies that any
  solution $u^*, \lambda_1^*, \lambda_2^*, \lambda_3^*$ of the KKT
  conditions with $\lambda_1^*\in\Fc_0$ satisfies
  \begin{align*}
    u^* = A(\lambda_1^*)(\lambda_2^* b_h^T + \lambda_3^* b_V^T
    -\hat{B}(\lambda_1^*) ). 
  \end{align*}
  Now the four different cases in the statement arise by applying the
  rest of the KKT conditions depending on whether the constraints
  $-b_h u-c_h \leq 0$, $-b_V u-c_V\leq0$ are active at the optimizer.
  The case $i=0$ corresponds to both constraints being inactive, the
  case $i=1$ to both constraints being active, the case
  $i=2$ to only the constraint $-b_V u-c_V\leq0$ being active, and $i=3$ to 
  only the constraint $-b_h u -c_h\leq 0$ being active.
  $\blacksquare$
\end{pf}

\begin{remark}\longthmtitle{Applicability of
    Proposition~\ref{prop:nec-suff-cond-compat}}\label{rem:applicability-second-suff-cond}
  {\rm
    Although the problem of knowing whether a nonlinear equation
    has any roots is undecidable in general, cf.~\citep{PSW:74}, if
    a root satisfying either of the specific conditions in
    Proposition~\ref{prop:nec-suff-cond-compat} can be rapidly
    found, this result provides a quick test for the compatibility
    of the two SOCCs in~\eqref{eq:socc-gp}.  A simple setting in
    which this holds is the following.  Recall that $\eta_1$ is a
    function of $x$
    and suppose that a root $\lambda_{x_0}^*$ of $\eta_1$ has been
    found at a point $x_0$. Moreover, suppose that the inequalities
    $\lambda_{2,1}(\lambda_{x_0}^*) > 0$,
    $\lambda_{3,1}(\lambda_{x_0}^*)>0$, $g_V(u_1^*(\lambda_{x_0}^*))<0$
    and $g_h(u_1^*(\lambda_{x_0}^*)) < 0$ are satisfied strictly.
    Then, under the assumptions of the Implicit Function
    Theorem~\cite[Theorem 2-12]{MS:95}, there exists a neighborhood
    $\Vc$ of $x_0$ such that for all $x\in\Vc$, there exists a root
    $\lambda_{x}^*$ of $\eta_1$ that is close to $\lambda_{x_0}^*$.
    Therefore, we can limit the search of the root to a neighborhood
    of $\lambda_{x_0}^*$ and we should expect to find a solution
    satisfying the conditions in
    Proposition~\ref{prop:nec-suff-cond-compat} fast. Analogous
    observations are valid for $i\in\{0,2,3\}$. }
  \demo
\end{remark}

\begin{remark}\longthmtitle{Necessity of
    Proposition~\ref{prop:nec-suff-cond-compat}}\label{rem:nec-of-second-suff-cond}
  {\rm Proposition~\ref{prop:nec-suff-cond-compat} is close to being
    a necessary and sufficient condition for compatibility. The gap
    arises from not including the cases where
    $\lambda_i^*\notin\Fc_i$ for $i \in \{0,1,2,3\}$
    or where the gradients of the active constraints
    in~\eqref{eq:opt-pb-licq} at the optimizer
    of~\eqref{eq:proof-nec-suf-cond-pb} are linearly dependent.  In
    these cases, a condition that ensures compatibility of the SOCCs
    can still be given on the basis of the KKT conditions
    of~\eqref{eq:proof-nec-suf-cond-pb}, but its statement becomes
    quite involved and we have not included it in
    Proposition~\ref{prop:nec-suff-cond-compat} for simplicity.  }
  \demo
\end{remark}


\begin{remark}\longthmtitle{Practical significance of sufficient
    conditions}\label{re:practical-app}
  {\rm Propositions~\ref{prop:small-errors}
    and~\ref{prop:nec-suff-cond-compat} are complementary to each
    other.  Proposition~\ref{prop:small-errors} requires the knowledge
    of the upper bound $B$, but is computationally cheap.
    Proposition~\ref{prop:nec-suff-cond-compat} requires less
    restrictive assumptions but involves finding a root of a nonlinear
    scalar equation, which can be more computationally expensive.
    Their practical usage is threefold.  First, if they are not met
    (which does not mean that the corresponding pair of SOCCs is not
    compatible), this can be taken as an indication that the estimates
    of the dynamics, CLF, and CBF need to be improved.  Therefore,
    Propositions~\ref{prop:small-errors}
    and~\ref{prop:nec-suff-cond-compat} pave the way for the design of
    active learning strategies that leverage them to decide when more
    data needs to be collected.  Second, these sufficient conditions
    can be used to identify the regions of the state space where
    compatibility might fail, and design control strategies that avoid
    them in order to guarantee recursive feasibility.  Third, given
    that in general, state-of-the-art SOCP solvers provide
    infeasibility and optimality certificates with the same time
    complexity, cf.~\citep[Section A]{AD-EC-SB:13}, our sufficient
    conditions can be used before solving the SOCP to save computation
    time in the case where the problem is unfeasible.  This latter
    point is illustrated in more detail in our simulations below,
    cf. Section~\ref{sec:sims}.  \demo }
\end{remark}

\section{Design and Regularity Analysis of Controllers Satisfying
  SOCCs
}\label{sec:regularity}

In this section, we study the existence and regularity properties of
controllers satisfying sets of SOCCs.
Our first result establishes that, if a set of state-dependent SOCCs
are strictly compatible, then there exists a smooth controller
satisfying all of them simultaneously. 

\begin{proposition}\longthmtitle{Existence of a smooth controller
    satisfying a finite number of SOCCs}\label{prop:existence-soccs}
  For $i\in\{1,\hdots,p\}$, let
  $Q_i:\real^n\to\real^{(m+1)\times m}, r_{i}:\real^{n}\to\real^{m+1},
  b_{i}:\real^{n}\to\real^m, c_{i}:\real^{n}\to\real$ be continuous
  functions on an open set~$\Gc \subset \real^n$.  If the $p$ SOCC
  inequalities $\norm{Q_i(x)u+r_i(x)}\leq b_{i}(x)u+c_{i}$,
  $i\in\{1,\hdots,p\}$, are strictly compatible on~$\Gc$, then there
  exists a $\Cc^{\infty}(\Gc)$ function $k:\Gc\to\real^{m}$ such that
  $\norm{Q_i(x)k(x)+r_i(x)}\leq b_{i}(x)k(x)+c_{i}(x)$ for all
  $i\in\{1,\hdots,p\}$ and all $x\in\Gc$.
\end{proposition}

This result is an extension of~\citep[Proposition~4.2.1]{PO:21} to a
finite set of SOCCs. Since SOCCs define convex sets, the proof follows
an identical argument and we omit it for space reasons.  The
combination of Propositions~\ref{prop:small-errors}
and~\ref{prop:existence-soccs} guarantees the smooth safe
stabilization of~\eqref{eq:control-affine-sys} under either worst-case
or probabilistic uncertainty.

\begin{corollary}\longthmtitle{Smooth safe
    stabilization under uncertainty}\label{cor:smooth-safe-stab}
  Let $\tilde{\Cc}$ be a neighborhood of $\Cc$, $h$ be an
  $\eta$-robust CBF, and assume~\eqref{eq:clf-slack}
  and~\eqref{eq:cbf-slack} are compatible on $\tilde{\Cc}$. Let $\Vc$
  be a neighborhood of the origin and $\tilde{\Vc}$ be the smallest
  sublevel set of $V$ containing $\Vc$.
    \begin{enumerate}
    \item\label{cor-smooth:first} \longthmtitle{Local smooth safe
        control} Suppose that~\eqref{eq:small-errors-compat}
      (resp.~\eqref{eq:small-vars}) holds at $x_0\in\Cc\backslash\Vc$
      and~\eqref{eq:socc-worst-case} (resp.~\eqref{eq:socc-gp}) is
      continuous at $x_{0}$. Then, there exists a neighborhood
      $\Wc_{x_0}$ of $x_{0}$, a smooth controller
      $k_{x_0}:\Wc_{x_0}\to\real^{m}$, and a time $t_{x_0}>0$ such
      that the flow map of $\dot{x}=f(x)+g(x)k_{x_0}(x)$, denoted by
      $\Psi_{t}(x)$, is such that $\Psi_{t}(x_0)\in\Cc$ and
      $\frac{d}{dt}V(\Psi_{t}(x_0))<0$ for all $t\in[0,t_{x_0})$
      (resp. with probability at least $1-2\delta$);
    \item\label{cor-smooth:second} \longthmtitle{Global smooth safe
        stabilization} Let $\hat{\Cc}$ be open with
      $\Cc\subseteq\hat{\Cc}\subseteq\tilde{\Cc}$.
      If~\eqref{eq:small-errors-compat} (resp.~\eqref{eq:small-vars})
      holds for all $x\in\hat{\Cc}\backslash\Vc$
      and~\eqref{eq:socc-worst-case} (resp.~\eqref{eq:socc-gp}) is
      continuous on $\hat{\Cc}\backslash\Vc$, then there exists a
      smooth controller
      $k:\text{int}( \hat{\Cc}\backslash\Vc ) \to \real^{m}$ such that
      all trajectories of $\dot{x}=f(x)+g(x)k(x)$ starting at $\Cc$
      remain in $\Cc$ and asymptotically converge to $\tilde{\Vc}$
      (resp. with probability at least $1-2\delta$);
  \end{enumerate}
\end{corollary}

\begin{remark}\longthmtitle{Asymptotic stability}
  {\rm If conditions~\eqref{eq:small-errors-compat}
    and~\eqref{eq:small-vars} hold for all points in
    $\tilde{\Cc}\backslash\{0\}$ (not only for all points in
    $\tilde{\Cc}\backslash\Vc$), then
    Corollary~\ref{cor:smooth-safe-stab}(ii) implies that the origin
    is asymptotically stable. This can only be the case if knowledge
    of $\nabla V$ near the origin is exact,
    cf. Remark~\ref{re:exact-Lyapunov-uncertainty}.  \demo}
\end{remark}

Note that the set $\Cc$ is unknown and hence checking the
conditions~\eqref{eq:small-errors-compat} and~\eqref{eq:small-vars}
for all $x\in\Cc\backslash\Vc$ may not be practical. This is the
reason why we introduce the set $\hat{\Cc}$ in
Corollary~\ref{cor:smooth-safe-stab}(ii).

Corollary~\ref{cor:smooth-safe-stab} establishes the existence of a
smooth safe stabilizing controller under uncertainty, but does not
provide an explicit closed-form design that can be used for
implementation. In what follows, we provide controller designs that
are explicit but have weaker regularity properties.
Let
\begin{align}\label{eq:st-sf-socp}
  u^*(x)&=\argmin{u\in\real^m}{\frac{1}{2}\norm{u}^2},
  \\
  \notag
  \textrm{s.t.} &\quad \norm{Q_i(x)u+r_i(x)}\leq b_i(x)u + c_i(x), \
                  i\in\{1,\hdots,p\}. 
\end{align}
Note that this program can be written as a second-order convex program
(SOCP), as shown in~\citep[Section 2.2]{FA-DG:03}.
%
%
If the constraints in~\eqref{eq:st-sf-socp} are
either~\eqref{eq:socc-worst-case} or~\eqref{eq:socc-gp}, we refer
to~\eqref{eq:st-sf-socp} as CLF-CBF-SOCP. The following result
establishes different conditions under which $u^{*}$ is point-Lipschitz and
locally Lipschitz.

\begin{proposition}\longthmtitle{Lipschitzness of SOCP
    solution}\label{prop:lipschitz-clf-cbf-socp} 
  Let $\{Q_i, r_i, b_{i},c_{i}\}_{i=1}^p$ be twice continuously
  differentiable at a point $x\in\real^n$ and assume the constraints
  in~\eqref{eq:st-sf-socp} are strictly compatible at $x$.
  Then $u^{*}$ is point-Lipschitz at $x$.  Further,
    for $i \in \until{p}$, let
  \begin{align*}
    g_i(x,u)
    & =\norm{Q_i(x)u+r_i(x)}-b_i(x)u-c_i(x),
    \\
    g_{i,1}(x,u)
    & = u^T (Q_i(x)^T Q_i(x)-b_i(x) b_i(x)^T)u+r_i(x)^2
    \\
    &\quad + 2(Q_i(x)^T r_i(x)-c_i(x) b_i(x))^T
      u-c_i(x)^2,
    \\
    g_{i,2}(x,u)& =-b_i(x)^T u-c_i(x),
  \end{align*}
  and define
  \begin{align*}
    \Ac(x) & := \setdef{i\in[p]}{\norm{Q_i(x)u^*(x)+r_i(x)}\neq 0, \
             g_i(x)\!=\!0},
    \\ 
    \Ac_1(x) & := \setdef{i\in[p]}{\norm{Q_i(x)u^*(x)+r_i(x)}= 0, \
               g_{i,1}(x)\!=\!0},
    \\ 
    \Ac_2(x) & := \setdef{i\in[p]}{\norm{Q_i(x)u^*(x)+r_i(x)}= 0, \
               g_{i,2}(x)\!=\!0}.
  \end{align*}
  Suppose that the vectors
  \begin{multline}\label{eq:G-G1-G2}
    \{ \nabla_u g_i(x,u^*(x)) \}_{i\in\Ac(x)} \cup \{ \nabla_u
    g_{i,1}(x,u^*(x)) \}_{i\in\Ac_1(x)}
    \\
    \cup \{ \nabla_u
    g_{i,2}(x,u^*(x)) \}_{i\in\Ac_2(x)} 
  \end{multline}
  are linearly independent.  Then, $u^*$ is locally Lipschitz
  at~$x$.
\end{proposition}
\begin{pf}
  First consider the points $x\in\Gc$ where
  $\norm{Q_i(x)u^*(x)+r_i(x)}\neq0$ for all $i\in[p]$. At these
  points, the constraints of~\eqref{eq:st-sf-socp} are twice
  continuously differentiable in $x$ and $u$ in a neighborhood of the
  optimizer. Moreover, since the constraints in~\eqref{eq:st-sf-socp}
  are strictly compatible, for any $\epsilon>0$ there exists
  $\hat{u}_{\epsilon}^{x}$ satisfying them strictly and such that
  $\norm{u^*(x)-\hat{u}_{\epsilon}^x}<\epsilon$.
  %
  %
  Since none of the
  constraints are active at $\hat{u}_{\epsilon}^x$, the
  Mangasarian-Fromovitz Constraint Qualification (MFCQ) holds at
  $\hat{u}_{\epsilon}^x$. By~\citep[Lemma 6.1]{GS:18} this implies that MFCQ
  also holds at $u^{*}(x)$. Furthermore, since the objective function
  in~\eqref{eq:st-sf-socp} is strongly convex and the constraints are
  convex, the second-order condition (SOC2)~\citep[Definition
  6.1]{GS:18} holds and by~\citep[Theorem 6.4]{GS:18}, $u^{*}$ is
  point-Lipschitz at $x$.  Next, consider any point $x\in\Gc$
  where 
  $\Ic_{x}=\setdef{i\in\{1,\hdots,p\}}{\norm{Q_i(x)u^*(x)+r_i(x)}=0}$
  is nonempty. Since the constraint
  $\norm{Q_i(x)u+r_i(x)}\leq b_{i}(x)u+c_{i}(x)$ is not differentiable
  at those points, we square the SOCCs
  in~\eqref{eq:st-sf-socp} associated to $\Ic_{x}$ to obtain the
  equivalent formulation with twice-continuously differentiable
  constraints:
  \begin{align}\label{eq:squared-st-sf-socp}
    u^*(x)& =\argmin{u\in\real^m}{\frac{1}{2}\norm{u}^2},
    \\
    \notag
    \textrm{s.t.} \quad & g_{i,1}(x,u)\leq0, \ g_{i,2}(x,u)\leq0, \
                          i\in\Ic_{x},
    \\ 
    \notag
    \quad & \norm{Q_i(x)u+r_i(x)}\leq b_i(x)u+c_i(x), \ i\in\{1,\hdots,p\}\backslash\Ic_{x},
  \end{align}
  Strict compatibility of the constraints in~\eqref{eq:st-sf-socp}
  implies the strict compatibility of the constraints
  in~\eqref{eq:squared-st-sf-socp} and, by the same argument as
  before, MFCQ holds at the optimizer. To show that SOC2 also holds
  for~\eqref{eq:squared-st-sf-socp}, note that the constraints
  $g_{i,1}(x,u)\leq0$ for $i\in\Ic_{x}$ cannot be active at the
  optimizer (otherwise, that would imply that
  $b_{i}(x)u^{*}(x)+c_{i}(x)=0$, implying that MFCQ is violated at the
  optimizer, reaching a contradiction). Thus, by the strict
  complementarity condition, the Lagrange multipliers
  associated with the constraints $g_{i,1}(x,u), i\in\Ic_{x}$ are zero
  and the Hessian of the Lagrangian $\Lc$
  of~\eqref{eq:squared-st-sf-socp} at the optimizer takes the form
  \begin{align*}
    &\nabla_u^2 \Lc(u^*, \{ \lambda_{i} \}_{i\in \Ac})_{x} = I_m +
    \sum_{i\in \Ac(x)}  {\lambda_{i}(x)} \nabla_u^2
    g_{i}(x,u^*(x)),
  \end{align*}
where $\lambda_{i}$ is the Lagrange multiplier associated with the
constraint $g_i(x,u)\leq0$ for $i\not\in\Ic_{x}$. Since
$\norm{Q_i(x)u^*(x)+r_i(x)}\neq0$ for the active constraints, their
Hessian is well-defined and is positive semidefinite due to their
convexity, making
$\nabla_u^2 \Lc(u^*, \{\lambda_{i}\}_{i\in J_0})_{x}$ positive
definite. Hence, SOC2 holds for~\eqref{eq:squared-st-sf-socp} at the
optimizer and, by~\citep[Theorem 6.4]{GS:18}, $u^{*}$ is
  point-Lipschitz at~$x$.
  Moreover, the assumption that
  the vectors in~\eqref{eq:G-G1-G2}
  are linearly
  independent implies that the gradients of the active constraints are
  linearly independent. By the same argument used to show that
  the SOC2 condition holds, the strong second-order sufficient condition
  also holds. This shows by~\citep[Theorem 4.1]{SMR:80} that $u^*$ is
  \textit{strongly regular} at $x$, which by~\citep[Corollary
  2.1]{SMR:80} implies that $u^*$ is locally Lipschitz at $x$.
  $\blacksquare$
\end{pf}

  Note that the
  reformulation~\eqref{eq:squared-st-sf-socp} in the proof of
  Proposition~\ref{prop:lipschitz-clf-cbf-socp} by squaring the
  constraints is done purely for analysis purposes and does not
  have to be done in practice when solving~\eqref{eq:st-sf-socp}.

  \begin{remark}\longthmtitle{Not-locally Lipschitz example without independence
      of gradients}\label{rem:non-lipschitz-licq} {\rm
      \cite{SMR:82} introduces an example of a
      parametric quadratic program with strongly convex objective
      function, smooth objective function and constraints, and for
      which Slater's condition holds for all values of the
      parameter. Moreover, the parametric optimizer of this problem is
      shown to be not locally Lipschitz.  Since the parametric QP presented by
      Robinson is a particular case of~\eqref{eq:st-sf-socp}, it also
      provides an example as to why the extra condition on the
      set~\eqref{eq:G-G1-G2}
      being linearly independent is necessary in order to guarantee
      local Lipschitzness of $u^*$.  The recent
      note~\citep{PM-AA-JC:23-scl} explores in detail the regularity
      properties of parametric optimization problems satisfying
      conditions similar to those of Robinson's counterexample.  }
    \demo
\end{remark}

As a consequence of Proposition~\ref{prop:lipschitz-clf-cbf-socp}, we
conclude that if the estimates $\hat{f}$, $\hat{g}$, $\hat{h}$,
$\widehat{\nabla h}$, $\hat{V}$, $\widehat{\nabla V}$ and worst-case
error bounds (resp. means and variances) that appear
in~\eqref{eq:socc-worst-case} (resp.~\eqref{eq:socc-gp}) are twice
continuously differentiable and the
conditions~\eqref{eq:small-errors-compat}
(resp.~\eqref{eq:small-vars}) hold, then the corresponding
CLF-CBF-SOCP controller is point-Lipschitz.
We also note that the condition that the vectors 
in~\eqref{eq:G-G1-G2}
are linearly independent corresponds to the Linear Independence
Constraint Qualification (LICQ)~\cite[Definition 2.4]{GS:18} for
problem~\eqref{eq:squared-st-sf-socp}.

Next we provide a formula, inspired by Sontag's universal
formula~\citep{EDS:89a}, for a smooth controller satisfying a single
SOCC defined by smooth functions.

\begin{proposition}\longthmtitle{Universal formula for a controller
    satisfying one SOCC}\label{prop:sontag-soccs}
  Let $l \in \mathbb{Z}_{>0}$ and assume
  $Q:\real^{n}\to\real^{(m+1)\times m}$, $r:\real^n\to\real^{m+1}$,
  $b:\real^n\to\real^m$, and $c:\real^n\to\real$ are $l$-continuously
  differentiable on an open set $\Gc\subseteq\real^{n}$.  Suppose that
  the SOCC $\norm{Q(x)u+r(x)}\leq b(x)u+c(x)$ is strictly feasible on
  $\Gc$ and $Q(x)^{T}Q(x)$ is invertible for all $x\in\Gc$. Let
  $\tilde{b}(x) = b(x)(Q^T(x)Q(x))^{-1} Q^{T}(x)$,
  $\tilde{c}(x) = c(x) - \tilde{b}(x) r(x)$,
  $\bar{b}(x) = (\norm{\tilde{b}(x)}-1)\norm{\tilde{b}(x)}$, and
  \begin{align}\label{eq:u-s-closed-form}
    v_{s}(x)=
    \begin{cases}
      0 & \text{if }  \norm{\tilde{b}(x)}\leq 1,
      \\
      \frac{-\tilde{c}(x)+\sqrt{\tilde{c}(x)^2 +
          \bar{b}(x)^2}}{\bar{b}(x)}\tilde{b}(x) & \text{if }  \norm{\tilde{b}(x)}> 1.
    \end{cases}
  \end{align}
  Further assume $v_{s}(x)-r(x)\in \text{Im}(Q(x))$ for all
  $x\in\Gc$.~Then
  \begin{align*}
    u_{s}(x):=(Q^T(x)Q(x))^{-1} Q^T(x)(v_s(x)-r(x)) ,
  \end{align*}
  is $l$-continuously differentiable for all $x\in\Gc$. Moreover,
  $\norm{Q(x)u_s(x)+r(x)}\leq b(x)u_{s}(x)+c(x)$ for all $x\in\Gc$.
\end{proposition}
\begin{pf}
  Let $v = Q(x)u+r(x)$. Since
  $Q^{T}(x)Q(x)$ is invertible and $\norm{Q(x)u+r(x)}\leq b(x)u+c(x)$
  is strictly feasible on $\Gc$, the resulting SOCC
  $\norm{v} \leq \tilde{b}(x)v+\tilde{c}(x)$ is also strictly feasible
  on $\Gc$. Moreover, $v_{s}$ satisfies it. Indeed, if
  $\norm{\tilde{b}(x)}\leq 1$, since the SOCC is feasible there exists $v^*$ such that $\norm{v^*}\leq \tilde{b}(x)v^{*}+\tilde{c}(x)$
  and it follows that $\tilde{c}(x)\geq0$. The case
  $\norm{\tilde{b}(x)}>1$ follows from a direct calculation.  If
  $\norm{\tilde{b}(x)}\neq 1$, $v_{s}$ is $\Cc^{l}$ at $x$ because
  $\tilde{b}$ and $\tilde{c}$ are $\Cc^{l}$ at $x$. If
  $\norm{\tilde{b}(x)}=1$, then $\tilde{c}(x)\neq0$ (otherwise, if
  $\tilde{c}(x)=0$, since the SOCC
  $\norm{v} \leq \tilde{b}(x)v+\tilde{c}(x)$ is strictly compatible,
  there would exist $\hat{v}$ such that
  $\norm{\hat{v}} < \tilde{b}(x)\hat{v} \leq \norm{\hat{v}}$, which is a
  contradiction).  Now, from the proof of~\citep[Theorem 1]{EDS:89a},
  the function
  \begin{align*}
    \phi(c,\alpha):=\begin{cases}
      0 & \text{if } \ \alpha\leq0,
      \\
      \frac{-c+\sqrt{c^2+\alpha^2}}{\alpha} & \text{else},
    \end{cases}
  \end{align*}
  is analytic at points of the form $(c,0)$, with $c\neq0$, so $v_{s}$ is $\Cc^{l}$ for all
  $x\in\Gc$. Moreover, since
  $v_{s}(x)-r(x)\in\text{Im}(Q(x))$ for all $x\in\Gc$, it also follows
  that $\norm{Q(x)u_s(x)+r(x)}\leq b(x)u_{s}(x)+c(x)$ for all
  $x\in\Gc$ and $u_{s}$ is $\Cc^{l}$ for all
  $x\in\Gc$.
  $\blacksquare$
\end{pf}

From the proof of Proposition~\ref{prop:sontag-soccs}, we observe that
in the case where the SOCC takes the form~\eqref{eq:socc-worst-case},
a simpler expression is available for a controller satisfying it. As a
result of Proposition~\ref{prop:sontag-soccs}, the proposed formula
can be used to guarantee safety or stability under the worst-case or
probabilistic uncertainties described in
Section~\ref{sec:rob-prob-safe-stab}.  

\begin{remark}\longthmtitle{Using the universal formula to filter a
    nominal controller}\label{rem:univ-filter} {\rm The universal
    formula in Proposition~\ref{prop:sontag-soccs} can also be used to
    render an existing nominal controller safe or stable.  Indeed, let
    $u_{\text{nom}}:\real^n\to\real^m$ be a nominal controller and
    define $\tilde{f}:\real^n\to\real^n$ as
    $\tilde{f}(x)=f(x)+g(x)u_{\text{nom}}(x)$ and the modified
    dynamics
  \begin{align}\label{eq:modified-nominal-dynamics}
    \dot{x}&=\tilde{f}(x)+g(x)\tilde{u},
  \end{align}
  with $\tilde{u}\in\real^m$.  By leveraging the estimates of $f$,
  $g$, $V$, and $h$ either in the worst-case or probabilistic case, we
  can formulate SOCCs similar to~\eqref{eq:socc-worst-case}
  and~\eqref{eq:socc-gp}, respectively, for the modified
  system~\eqref{eq:modified-nominal-dynamics}. Depending on which SOCC
  we choose, this allows us to use the universal formula in
  Proposition~\ref{prop:sontag-soccs} to obtain a safe or a stable
  controller $\tilde{u}_s:\real^n\to\real^m$, which in turn results in
  $u_{\text{fi-nom}}(x) = u_{\text{nom}}(x)+\tilde{u}_s(x)$ being a
  safe or a stable controller for~\eqref{eq:control-affine-sys}. We
  refer to this controller $u_{\text{fi-nom}}$ as the filtered version
  of the nominal controller~$u_{\text{nom}}$. This generalizes the
  safe filtering of a nominal controller in the uncertainty-free case,
  cf.~\cite{LW-ADA-ME:17,ADA-SC-ME-GN-KS-PT:19}. } \demo
\end{remark}



\section{Simulations}\label{sec:sims}
In this section we illustrate our results in an example. For
simplicity, we focus on the case of worst-case error estimates.
Consider a control-affine planar system of the
form~\eqref{eq:control-affine-sys} with
  $f(x,y)=(-x,-(x^2+5)y)$ and $g(x,y)=(1,0.1)$.
We consider the CBF $h(x,y)=x^{2}+(y-4)^{2}-4$.

\emph{From data to estimates and error bounds:} We obtain here
worst-case error models, cf.  Section~\ref{sec:rob-prob-safe-stab},
from data. For simplicity, we assume that the CLF
$V(x,y)=\frac{1}{2}(x^2+y^2)$ is known, so that
$\widehat{\nabla V}=\nabla V$ and $\hat{V}=V$. We also assume that the
obstacle is known to be a circle with center at $(0, 4)$, but its
radius is uncertain, so that $\hat{h}(x,y)=x^{2}+(y-4)^{2}-3.8$, and
$\widehat{\nabla h}=\nabla h$, $e_{h}=0.4$, $e_{\nabla h}=0$. We have
access to an oracle that, given a query point $(x,y)\in\Cc$, returns
noiseless measurements $(f(x,y),g(x,y))$ of the functions
in~\eqref{eq:control-affine-sys} (the noisy case can be considered
without major modifications).  Given a set of $N$ measurements
$\Dc=\{(x_i, y_i),f(x_i, y_i),g(x_i, y_i)\}_{i=1}^{N}$ obtained by
querying the oracle, we estimate $f$ at $(x,y)\in\real^{2}$ as
$\hat{f}(x,y)=f(p_{\text{cl}}(x,y))$, where $p_{\text{cl}}(x,y)$ is
the closest datapoint to $(x,y)$.  Prior knowledge of (not necessarily
tight) Lipschitz constants of $f$ and $g$ in a compact region
containing the origin, the initial conditions and
$\{ (x_i,y_i) \}_{i=1}^N$ ($K_f =28.0$ and $K_g =3.2$
  respectively) is also available. We compute the corresponding
worst-case error bounds as
$e_{f}(x,y):=K_{f}\norm{(x,y)-p_{\text{cl}}(x,y)}$.  We do similarly
for $\hat{g}$ and~$e_{g}$.

\begin{figure*}[htb]
  \centering
  {\includegraphics[width=0.8\linewidth]{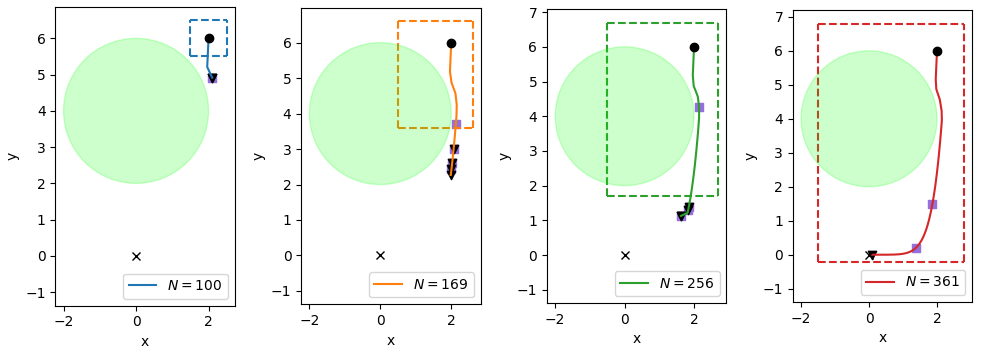}}
  \caption{Safe stabilization of a planar system with worst-case
    uncertainty error bounds. The green ball is the unsafe set and
    black dots denote initial conditions. Dashed lines enclose the
    region where data is located for different $N$.  Solid lines show
    the evolution under the corresponding CLF-CBF-SOCP controller
    in~\eqref{eq:st-sf-socp}. Black triangles indicate points where
    the sufficient conditions for
    feasibility~\eqref{eq:small-errors-compat} in
    Proposition~\ref{prop:small-errors} do not hold.
      Purple squares denote points where
      the root-finding method (\texttt{fsolve} from Python's \texttt{SCIPY} library) 
      did not return a solution satisfying the
      sufficient condition of
      Proposition~\ref{prop:nec-suff-cond-compat}.}
  \label{fig:sim-offline}
\end{figure*}

\emph{Performance dependency on error estimates:} Here we illustrate
how smaller estimation errors lead to improved performance. We use
different datasets with different number of data points $N$ to
generate $\hat{f}$, $\hat{g}$, $e_{f}$, and $e_{g}$. We solve the
resulting CLF-CBF-SOCP every $0.01$s with initial condition at
$(2.0, 6.0)$ and plot the trajectories until it becomes unfeasible. We
compare the results for different $N$ in
Figure~\ref{fig:sim-offline}. Larger datasets with data from a
neighborhood of the origin allow trajectories to converge closer to
the origin before the problem becomes unfeasible. This illustrates one
of the critical points of the paper: optimization-based control
formulations that take uncertainty into account in order to ensure
safety or stability might be unfeasible depending on the specific
system and the magnitude of the errors in the employed
approximations. Our results here provide quantifiable conditions to
determine whether the accuracy of the approximations is sufficient or,
instead, they need to be refined in order to guarantee feasibility.
In the plot, we observe that the sufficient conditions in
Propositions~\ref{prop:small-errors}
and~\ref{prop:nec-suff-cond-compat} serve as a good indicator of
when the SOCP actually becomes unfeasible, hence illustrating how
they can be used to infer when the available estimates are
insufficient to guarantee that the controller is well defined.
%
%


\emph{Online safe stabilization:} We illustrate also the case where
data is collected online. We start from an initial set of 150
measurements of $f$, $g$ and $h$ near the initial condition obtained
by querying the oracle.  Given an initial condition, at
  every $0.01$s we check whether the conditions
  in~\eqref{eq:small-errors-compat} hold. If this is the case, we find
  the CLF-CBF-SOCP controller and execute it.  If during the
execution the conditions in~\eqref{eq:small-errors-compat} stop being
satisfied at some point $\bar{x}$, we query the oracle to obtain
measurements of $f$ and $g$ at $\bar{x}$ (making it feasible) and a
small neighborhood around it (for improved performance).
Figure~\ref{fig:sim-online} illustrates executions of this procedure
for three different initial conditions.  As trajectories
approach the origin, more measurements need to be taken because the
conditions in~\eqref{eq:small-errors-compat} become harder to satisfy.

\begin{figure}[htb]
  \centering
  {\includegraphics[width=0.6\linewidth]{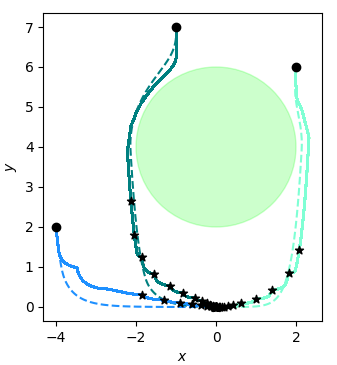}}
  \caption{
    Safe stabilization of a planar system with worst-case
    uncertainty error bounds. The green ball is the unsafe set and
    black dots denote initial conditions.
    The solid lines display the evolution of the controller obtained
    by solving the CLF-CBF-SOCP~\eqref{eq:st-sf-socp}. Black stars
    denote points where measurements have been taken. All trajectories
    asymptotically converge to a ball around the origin of radius
    0.01.  For reference, the dashed lines display the evolution of a
    min-norm controller with perfect knowledge of the dynamics, CBF
    and CLF (CLF-CBF QP)~\citep{ADA-XX-JWG-PT:17}, for which the
    trajectories stay safe and asymptotically converge to the origin.  }
  \label{fig:sim-online}
\end{figure}


\emph{Time complexity:} We show here the computational
  savings of checking the sufficient conditions in
  Propositions~\ref{prop:small-errors}
  and~\ref{prop:nec-suff-cond-compat} as compared to directly solving
  the SOCP using the Embedded Conic Solver from the Python library
  CVXPY.  Figure~\ref{fig:time-duration} shows that the time
  complexity of using the SOCP solver is higher than the time
  complexity of checking the sufficient condition in
  Proposition~\ref{prop:nec-suff-cond-compat}, which is in turn higher
  than the time complexity of checking the sufficient condition of
  Proposition~\ref{prop:small-errors}.  Since, in general,
  state-of-the-art SOCP solvers provide infeasibility and optimality
  certificates with the same time complexity, cf.~\citep[Section
  A]{AD-EC-SB:13}, our sufficient conditions can be used to save
  computation time in the case where the problem is unfeasible, cf.
  Remark~\ref{re:practical-app}.

\begin{figure}[htb]
  \centering
  {\includegraphics[width=0.99\linewidth]{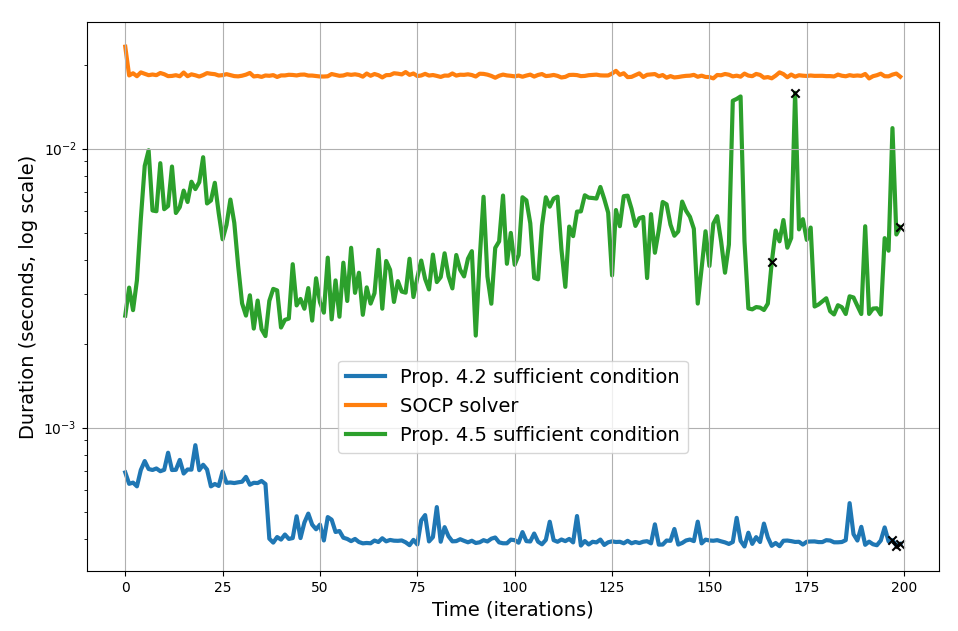}}
  \caption{{\color{blue} Time complexity comparison of evaluating the sufficient
    conditions in Proposition 4.2 (blue), Proposition 4.5 (green)
    and using an SOCP solver (orange) along the trajectory with $N=361$ in
    Figure~\ref{fig:sim-offline}. The black crosses denote points where the
    sufficient conditions do not hold. On average, the sufficient
    conditions in Proposition 4.2 is 50 times faster to evaluate
    than directly solving the SOCP.}
  }
  \label{fig:time-duration}
\end{figure}

\section{Conclusions}
We have studied conditions to ensure the safe stabilization of a
nonlinear affine control system under uncertainty.  Given either
worst-case or probabilistic estimates of the dynamics, CBF and CLF,
SOCCs encode the impact of uncertainty on the ability to guarantee
stability and safety. We have provided conditions for
  the compatibility of the relevant pairs of SOCCs and provided
  explicit bounds on the error estimates that ensure these SOCCs are
  compatible. We have built on these results to ensure the existence
of a smooth safe stabilizing controller, to show the
  point-Lipschitz and locally Lipschitz regularity of the min-norm
CLF-CBF-SOCP-based controller, and to prove the regularity of a
universal controller for the satisfaction of a single SOCC.  Future
work will characterize the conditions for compatibility in terms of
data, design online safe stabilization mechanisms that balance
computational effort, sampling rate, and performance using
resource-aware control, explore the design of universal formulas for
more than one SOCC, and implement our results on physical testbeds.


{\small
\bibliography{../bib/alias,../bib/JC,../bib/Main-add,../bib/Main}
\bibliographystyle{plainnat}
}

\end{document}